\documentclass[envcountsame,runningheads]{llncs}
\usepackage{graphics}
\usepackage{tikz,pseudocode,url,tikz-qtree}
\usetikzlibrary{shadows}
\usepackage{amssymb} 
\newtheorem{observation}{Observation}

\usepackage{soul,tikz, adjustbox}
\usepackage{forest, tikz-qtree,pseudocode,ulem, framed}
\newsavebox{\tempbox}
\newcommand{\cbox}[2]{%
  \fcolorbox{black}{#1}{\texttt{#2\strut}}\kern-\fboxrule}
\usepackage{times}
\usepackage{palatino}
\usepackage{amssymb,amsmath}
\usepackage{epsfig}
\usepackage{xcolor}

\begin{document}

\title{A Language-Theoretic Approach to the Heapability of Signed Permutations}           
{}                 

\author{Gabriel Istrate}
\institute{Faculty of Mathematics and Computer Science, University of Bucharest\\
Str. Academiei 14, 
011014, Sector 1, 
Bucharest, Romania\\
\email{gabriel.istrate@unibuc.ro}
}

\maketitle

\begin{abstract}

We investigate a signed version of the Hammersley process, a discrete process on words related to a property of integer sequences called heapability (Byers et al., ANALCO 2011). The specific version that we investigate corresponds to a version of this property for signed sequences.

We give a characterization of the words that can appear as images the signed Hammersley process. In particular we show that the language of such words is the intersection of two deterministic one-counter languages.
\end{abstract}


\section{Introduction}

Consider the following process $H_k$ that generates strings over the alphabet $\Sigma_{k}=\{0,1,\ldots, k\}$, $k\geq 1$: start with the empty word $w_0=\lambda$. Given word $w_n$, to obtain $w_{n+1}$ insert a $k$ at an arbitrary position of $w_n$ (we will regard the new $k$ as a "particle with $k$ lives").  If $w_n$ contained at least one nonzero letter to the right of the newly inserted $k$ then decrease by 1 the leftmost such letter (that is the new particle takes one life from the leftmost live particle to its right). The process has been introduced in \cite{istrate2015heapable}, related to a special property of integer sequences called \textit{heapability} (see 
\cite{byers2011heapable} for a definition, and \cite{istrate2015heapable,heapability-thesis,istrate2016heapability,basdevant2016hammersley,basdevant2017almost,balogh2017computing,hammersley-rochian} for subsequent work), and has been called (by analogy to the classical case $k=1$, which it generalized) \textit{the Hammersley process}. This process has proved essential \cite{aldous1995hammersley} in the scaling analysis of the longest-increasing subsequence of a random permutation, one of the celebrated problems in theoretical probability \cite{romik2015surprising}, and was subsequently rigorously analyzed in \cite{basdevant2016hammersley,basdevant2017almost}. 

In \cite{hammersley-rochian} we have undertaken a language-theoretic approach to the study of process $H_k$ by characterizing the language $L(H_k)$ of possible words in the process $H_k$. It was shown that the language $L(H_k)$ is regular for $k=1$ and context-free but nonregular for $k\geq 2$. We also gave an algorithm to compute the formal power series associated to the language $L(H_k)$ (where the coefficient of each word $w$ is equal to its multiplicity in the process $H_k$). We attempted to use this algorithm to study the typical large-scale behavior of process $H_k$ (while this study provided valuable information, the convergence to the limit behavior turned out to be fairly slow). 

The purpose of this paper is to study (with language-theoretic tools similar to those in \cite{hammersley-rochian}) a variant of the Hammersley process that we will call \textit{the signed Hammersley process}, and we will denote by $H_{k}^{sign}$. This is a process over the alphabet $\Gamma_{k}:=\{0^{+},0^{-},1^{+},1^{-}, 2^{+},2^{-}, \ldots, k^{+},k^{-}\}$ and differs from process $H_k$ in the following manner: 
\begin{itemize} 
\item The newly inserted letter will always be a \textit{signed} version of $k$, that is, it will either be a $k^{+}$ or a $k^{-}$.
\item  When a $k^{-}$ is inserted it subtracts 1 from the closest (if any) nonzero letter to its right having positive polarity. That is, inserting $k^{-}$ turns a $k^{+}$ into a $(k-1)^{+}$, a $(k-1)^{+}$ into a $(k-2)^{+}$, $\ldots$, a $1^{+}$ into a $0^{+}$, but has no effect on letters of type $0^{+}$ or letters with negative polarity.  
\item On the other hand, inserting a $k^{+}$ subtracts a 1 from the closest digit with negative polarity and nonzero value (i.e. one in  $\{1^{-}, \ldots, k^{-}\}$), if any, to its right. 
\end{itemize} 

\begin{example} 
A depiction of the first few possible steps in the evolution of the ordinary Hammersley tree process is presented in Figure~\ref{fig-1}(a). Similarly for the signed Hammersley tree process, see Figure~\ref{fig-1}(b). 
\end{example} 

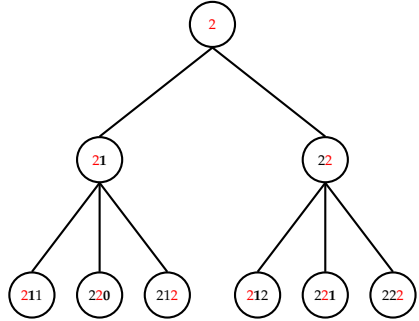
\begin{figure} 
\begin{center} 

\begin{tikzpicture}[
    scale=0.6, transform shape, thick,
    every node/.style = {draw, circle, minimum size = 10mm},
    grow = down,  
    level 1/.style = {sibling distance=5cm},
    level 2/.style = {sibling distance=1.5cm}, 
    level 3/.style = {sibling distance=1cm}, 
    level distance = 3cm
  ]

\node[shape=circle] {\textcolor{red}{2}}
    child{ node [shape=circle] {\textcolor{red}{2}{\bf 1}} 
    child{ node [shape=circle] {\textcolor{red}{2}{\bf 1}1}}
         child{ node [shape=circle] {2\textcolor{red}{2}{\bf 0}}}
         child{ node [shape=circle] {21\textcolor{red}{2}}}	
            }                 
    child{ node [shape=circle] {2\textcolor{red}{2}}
    	child{ node [shape=circle] {\textcolor{red}{2}{\bf 1}2}}
         child{ node [shape=circle] {2\textcolor{red}{2}{\bf 1}}}
         child{ node [shape=circle] {22\textcolor{red}{2}}}
         }
		
; 
\end{tikzpicture}
\end{center} 
\label{fig-1}
\caption{(a). Words in the binary Hammersley process ($k=2$). The first node $\lambda$ is omitted. Insertions are in red.  Positions that lost a life at the current stage are bolded.} 
\end{figure} 
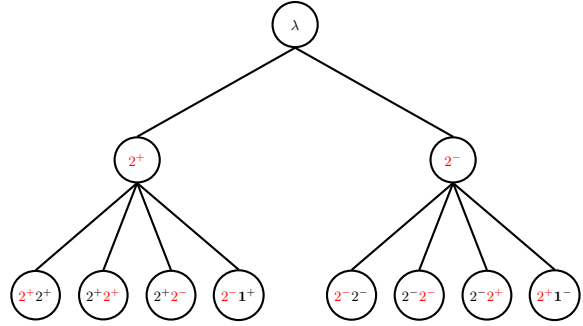
\begin{figure} 
\begin{center} 
\begin{tikzpicture}[
    scale=0.6, transform shape, thick,
    every node/.style = {draw, circle, minimum size = 10mm},
    grow = down,  
    level 1/.style = {sibling distance=7cm},
    level 2/.style = {sibling distance=1.5cm}, 
    level 3/.style = {sibling distance=1cm}, 
    level distance = 3cm
  ]

\node[shape=circle] {$\lambda$}
    child{ node [shape=circle] {$\textcolor{red}{2^{+}}$} 
    child{ node [shape=circle] {$\textcolor{red}{2^{+}}2^{+}$}}
    child{ node [shape=circle] {$2^{+}\textcolor{red}{2^{+}}$}}
         child{ node [shape=circle] {$2^{+}\textcolor{red}{2^{-}}$}}
         child{ node [shape=circle] {$\textcolor{red}{2^{-}}\mathbf{1^{+}}$}}	
            }                 
    	child{ node [shape=circle] {\textcolor{red}{$2^{-}$}} 
    child{ node [shape=circle] {$\textcolor{red}{2^{-}}2^{-}$}}
    child{ node [shape=circle] {$2^{-}\textcolor{red}{2^{-}}$}}
         child{ node [shape=circle] {$2^{-}\textcolor{red}{2^{+}}$}}
         child{ node [shape=circle] {$\textcolor{red}{2^{+}}\mathbf{1^{-}}$}}	
            }	
; 
\end{tikzpicture}
\end{center}
\caption{Words in the binary signed Hammersley tree process. Insertions are in red.  Positions that lost a life at the current stage are bolded.} 
\label{fig-1-prime}
\end{figure}

In this paper we provide a complete characterization of the language of words in the signed Hammersley  process (see Definition~\ref{def-1} for the definition of $k$-dominant strings): 

\begin{theorem} A word $z\in \Gamma_k^{*}$ is generated by the signed Hammersley process if and only if $z$ and all its nonempty prefixes are $k$-dominant. 
\label{thm-main} 
\end{theorem} 

\begin{corollary} For every $k\geq 1$ if $L(H_{k}^{sign})$ is the language of words that appear in the signed Hammersley process there exist two deterministic context-free languages (in fact $L_1,L_2$ are even deterministic one-counter languages, see \cite{valiant1975deterministic}) s.t. 
$L(H_k^{sign})=L_1\cap L_2$. 
\label{cor-1}
\end{corollary} 

Not surprisingly, the motivation behind our study is a constrained variant of heapability for signed sequences (permutations) $(\sigma(1),\sigma(2),\ldots, \sigma(n))$. Briefly, in the setting we consider every integer $\sigma(i)$ in the sequence comes with a sign $\tau(i)\in \{\pm 1\}$. An integer $\sigma(i)$ can only become the child of an integer $\sigma(j)$ in a heap-ordered tree when $\sigma(j)<\sigma(i)$ and $\sigma(i)$ has the opposite sign (i.e. $\tau(i)=-\tau(j)$). When the signed permutation cannot be inserted into a single heap-ordered tree we are, instead, concerned with inserting it into \textrm{the smallest possible number} of heap-ordered trees. 

The outline of the paper is as follows: In Section 2 we review some preliminary notions we will need in the sequel. Our main result is presented in Section 3. We then give (Section 4) an algorithm for computing the formal power series associated to the signed Hammersley process. Later sections are no longer primarily string-theoretic, and instead attempt to explain the problem on signed integer sequences that motivated our results: in Section 5 we define the heapability of signed permutations and show that a greedy algorithm partitions a signed permutation into a minimal number of heap-ordered trees. This motivates and explains the definition of the signed Hammersley process, that describes the dynamics of "slots" generated by this algorithm. 

We refrain in this paper, though, from invastigating the scaling in the Ulam-Hammersley problem for signed permutations, and leave this for future work. 
\section{Preliminaries}

We will assume general acquaintance with formal language theory, as presented in, say, \cite{harrison1978introduction}, and its extension to (noncommutative) formal power series. For this last topic useful (but by no means complete) references are \cite{salomaa2012automata,berstel2011noncommutative}. 
Denote by $\Gamma_{k}$ the alphabet $\{0^{+},0^{-},\ldots, k^{+},k^{-}\}$. As usual, for $z\in \Gamma_{k}^{*}$ and $p\in \Gamma_k$, denote by $|z|$ the length of $z$ and 
 by $|z|_{p}$ the number of appearances of letter $p$ in $z$. Also, for $j\in 0,\ldots k$ define $|z|_j=|z|_{j^{+}}+|z|_{j^{-}}.$

\begin{definition} 
\label{def-1}
String $z\in \Gamma_{k}^{*}$ is called \textit{$k$-dominant} iff it starts with a letter from the set $\{k^{+},k^{-}\}$, and the following two conditions are satisfied: 
\begin{equation} 
|z|_{k^{+}}-\sum\limits_{i= 1}^{k} i\cdot |z|_{(k-i)^{-}}+\sum_{i=0}^{k-1} |z|_{i^{+}}\geq 0
\label{ineq-1}
\end{equation} 
and 
\begin{equation} 
|z|_{k^{-}}-\sum\limits_{i=1}^{k} i\cdot |z|_{(k-i)^{+}}+\sum\limits_{i=0}^{k-1} |z|_{i^{-}}\geq 0
\label{ineq-2}
\end{equation} 
at least one of the inequalities being strict, namely the one that corresponds to the first letter of $z$.  
\end{definition} 

\begin{example} 
The only words $z\in \Gamma_{k}^{1}$ that are $k$-dominant are $z=k^{+}$ and $z=k^{-}$. 
\end{example} 

\begin{example} For $z\in \{0^{+},0^{-},k^{+},k^{-}\}^{*}$ (or $k=1$) the two conditions in Definition~\ref{def-1} become $|z|_{k^{+}}-k\cdot |z|_{0^{-}}+|z|_{0^{+}}\geq 0$ 
and $|z|_{k^{-}}-k\cdot |z|_{0^{+}}+|z|_{0^{-}}\geq 0$. 
\end{example} 

\begin{definition} 
A formal power series with nonnegative integer coefficients is a function $f:\Gamma_{k}^{*}\rightarrow \mathbf{N}$. We will denote by ${\bf N}(<\Gamma_{k}>)$ the set of such formal power series. 
\end{definition} 

\begin{definition} Given $k\geq 1$, {\rm the signed Hammersley power series of order $k$} is the formal power series 
  $F_{k}\in {\bf N}(<\Gamma_{k}>)$ that counts the multiplicity of a given word $w\in \Gamma_{k}^{*}$ in the signed Hammersley process of order $k$.  

  The {\rm signed Hammersley language of order $k$,} $L(H_k^{sign})$, is defined to be the support of $F_{k}$, i.e. the set of words $w\in \Gamma_{k}^{*}$ that are outputs of the signed Hammesley process of order $k$. 
\end{definition}

\section{Main Result: The language of the process $H_{k}^{sign}$}

The proof of Theorem~\ref{thm-main} proceeds by first showing that every word generated by the signed Hammersley process satisfies the condition in the theorem, and conversely, every string that satisfies the conditions can be generated by the process. For the first claim we need a couple of simple lemmas: 

\begin{lemma} \label{claim:startWith2}
Every word in $L(H_k^{sign})$ starts with a $k^{+}$ or a $k^{-}$. 
\end{lemma}

\begin{proof}
Digits are only modified by insertions to their left. So the leftmost $k^{+}$ or $k^{-}$ remains unchanged. 
\end{proof}\qed

\begin{lemma}\label{claim:prefixIsWord}
$L(H_k^{sign})$ is closed under prefix.
\end{lemma}

\begin{proof}
Given $w\in L(H_k^{sign})$ and a position $p$ in $w$, just ignoring all insertions to the right of $p$ yields a construction for the prefix of $w$ ending at $p$. 
\end{proof}\qed 

\begin{lemma}\label{claim:hasMore2}
Every word in $L(H_k^{sign})$ is $k$-dominant.  
\end{lemma}

\begin{proof}
Let us count the number of particles of type $(k-i)^{+}$ in $z$, $i\geq 1$. Such a particle arises from a $k^{+}$ particle through a chain 
\[
k^{+}\stackrel{k^{-}}{\rightarrow} (k-1)^{+} \stackrel{k^{-}}{\rightarrow} \ldots \stackrel{k^{-}}{\rightarrow} (k-i)^{+}
\]
requiring $i$ particles of type $k^{-}$ and killing one particle of type of type $k^{+}$. 
Similarly, a particle of type $(k-i)^{-}$, $i\geq 1$, arises from a $k^{-}$ particle through a chain 
\[
k^{-}\stackrel{k^{+}}{\rightarrow} (k-1)^{-} \stackrel{k^{+}}{\rightarrow} \ldots \stackrel{k^{+}}{\rightarrow} (k-i)^{-}
\]
requiring $i$ particles of type $k^{+}$ and killing one particle of type of type $k^{-}$. 
There are also $\lambda^{+}$ particles of type $k^{+}$ that don't kill any particle, and similarly $\lambda^{-}$ particles of type $k^{+}$ that don't kill any particle. 

Thus the number of particles of type $k^{+}$ is 
\begin{equation} 
|z|_{k^{+}}=\lambda^{+}+\sum_{i\geq 1} i\cdot |z|_{(k-i)^{-}}-\sum_{i\geq 1} |z|_{(k-i)^{+}}
\label{ineq-3}
\end{equation} 
and similarly 
\begin{equation} 
|z|_{k^{-}}=\lambda^{-}+\sum_{i\geq 1} i\cdot |z|_{(k-i)^{+}}-\sum_{i\geq 1} |z|_{(k-i)^{-}}
\label{ineq-4} 
\end{equation} 
Putting the condition $\lambda^{+},\lambda^{-}\geq 0$ (the corresponding one being strictly $>0$) we infer that $z$ is $k$-dominant. 
\end{proof}\qed

From Lemmas~\ref{claim:startWith2},~\ref{claim:prefixIsWord} and~\ref{claim:hasMore2} it follows that any word $z\in L(H_k^{sign})$ satisfies the conditions~(\ref{ineq-1}) and~(\ref{ineq-2}) in the theorem. 

We show the converse as follows:

\begin{lemma} \label{lemma:semiInL2} Every  
word $z$ satisfying conditions in Theorem~\ref{thm-main} is an output of the signed Hammersley process. 
\end{lemma}
\begin{proof} 
Assume otherwise. Consider a  $z$ of smallest length that is not the output of the signed Hammersley process. Clearly $|z|>1$, since $z = k^{+}$ and $z=k^{-}$ are outputs.  
 
 Consider the last occurrence $l$ of one of the letters $k^{+},k^{-}$ in $z$. Without loss of generality assume that $l=k^{+}$ (the other case is similar). 
 
 If to the right of $l$ one had only positive letters (if any) then consider the word $w_1$ obtained by pruning the letter $l$ from $z$. Its prefixes $p$ are either prefixes of $z$ (hence $k$-dominant) or are composed of the prefix $w_0$ of $z$ cropped just before $l$ plus some more positive letters. For such a prefix we have to verify the two conditions~(\ref{ineq-1}) and~(\ref{ineq-2}) are satisfied: The second one follows directly from condition~(\ref{ineq-2}) for 
 the word $z$, since all letters of $z$ with negative polarity are present in $p$. As for the first one, it follows from the fact that condition~(\ref{ineq-1}) is valid for $w_0$, since some more positive letters are added. In conclusion, all prefixes of $w_1$ are $k$-dominant. 
As $|w_1|=|z|-1$, by the minimality of $z$, it follows that $w_1$ is an output of the signed Hammersley process. But then the process can output $z$ by simply simulating the construction for $w_1$ and then inserting the last $k^{+}$ of $z$ into $w_{1}$ in  its proper position. 

The remaining case is when $l=k^{+}$ has some negative letters to its right. Assume that $s$ is the first letter of negative type to the right of $l$. $s$ cannot be $k^{-}$, since we already saw the last of $k^{+},k^{-}$ at position $l$. Let $w_2$ be the word obtained by removing $l$ from $z$ and increasing the value of $s$ by 1. We claim that $w_2$ and all its prefixes are $k$-dominant.

\begin{equation}
|w_2|_{k^{+}}-\sum\limits_{i= 1}^{k} i\cdot |w_2|_{(k-i)^{-}}+\sum_{i=0}^{k-1} |w_2|_{i^{+}}=|z|_{k^{+}}-\sum\limits_{i= 1}^{k} i\cdot |z|_{(k-i)^{-}}+\sum_{i=0}^{k-1} |z|_{i^{+}} 
\label{ineq-10}
\end{equation} 
In fact this is true for every prefix of $w_2$, since losing a $k^{-}$ is not counted, and adding one doesn't change the sum. 

\begin{equation} 
|w_2|_{k^{-}}-\sum\limits_{i=1}^{k} i\cdot |w_2|_{(k-i)^{+}}+\sum\limits_{i=0}^{k-1} |w_2|_{i^{-}}=|z|_{k^{-}}-\sum\limits_{i=1}^{k} i\cdot |z|_{(k-i)^{+}}+\sum\limits_{i=0}^{k-1} |z|_{i^{-}} \geq 0
\label{ineq-11}
\end{equation} 
Again, this is true for every prefix of $z$ containing $s$, since losing $k^{-}$ is compensated by increasing $s$, so the sum doesn't change.  

Because of the minimality of $z$, $w_2$ is an output of the signed Hammersley process. But then we can obtain $z$ as an output of this process by first simulating the construction of $w_2$ and then inserting $l$ (which corrects the value of $s$ as well). 
\end{proof}\qed

\subsection{Proof of Corollary~\ref{cor-1}}

\begin{proof}

The claim that $L(H_k^{sign})$ is a the intesection of two deterministic one-counter languages follows from Theorem~\ref{thm-main} as follows: We construct two one-counter PDA's, $A_1$ and $A_2$, both having input alphabet $\Gamma_k$. Each of them enforces one of the conditions~(\ref{ineq-1}) and~(\ref{ineq-2}), respectively, including the fact that the appropriate inequality is strict. 
We describe $A_1$ in the sequel, the functioning of $A_2$ is completely analogous. 

The stack alphabet of $A_1$ comprizes two stack symbols, an effective symbol $*$ and the bottom symbol  $Z$. Transitions are informally specified in the following manner: 
\begin{itemize}
\item[-] $A_{1}$ starts with the stack consisting of the symbol $Z$. If the first symbol it reads is neither $k^{+}$ nor $k^{-}$,  $A_{1}$ immediately rejects. 
\item[-] $A_1$ has parallel (but disjoint) sets of states corresponding to the situations that the first letter is a $k^{+}$, respectively a $k^{-}$. 
\item[-] The first alternative requires that the difference in~(\ref{ineq-1}) is strictly positive. $A_1$ enforces this by first ignoring the first letter and then enforcing the fact that the difference in~(\ref{ineq-1}) \textit{corresponding to the suffix obtained by dropping the first letter} is always $\geq 0$. This is similar to the behavior in the case when the first letter is a $k^{-}$, which we describe below. 
\item[-] when reading any subsequent positive symbol, $A_{1}$ pushes a $*$ on stack. 
\item[-] if the next symbol of the input is $(k-i)^{-}$, $i\in 1\ldots k$, $A_{1}$ tries to pop $i$ star symbols from the stack. If this ever becomes impossible (by reaching $Z$), $A_{1}$ immediately rejects. 
\item[-] When reaching the end of the word $A_{1}$ accepts. 
\end{itemize}

\end{proof}\qed

\begin{observation} 
A nagging question that would seem to be a simple exercise in formal language theory (but which so far has escaped us) is proving that $L(H_{k}^{sign})$ is not a context free language. 

This is, we believe, plausible since the words in $L(H_{k}^{sign})$ have to satisfy not one but two inequalities,~(\ref{ineq-1}) and~(\ref{ineq-2}) that constrain their Parikh distributions. Verifying such inequalities would seemingly require two stacks. 

We have attempted (but failed) to prove this by applying Ogden's lemma. So we leave this as an open question. 
\end{observation} 

\section{Computing the formal power series of the signed Hammersley process}

In this section we study the following problem: given a word $z\in \Gamma_k^{*}$, compute the number of copies of $z$ generated by the signed Hammersley process. The problem is, of course, a generalization of the one in the previous section: we are interested not only if a word can/cannot be generated, but in how many ways. 

The real motivation for developing such an algorithm is  its intended use (similar to the use of the analog algorithm from the unsigned case in \cite{byers2011heapable}) to attempt to analyze the scaling in the Ulam-Hammersley problem for heap decompositions, and is beyond the scope of the current paper. Roughly speaking, by studying the multiplicity of words in the signed Hammersley process we hope to be able to exactly sample from the distribution of words for large $n$ and estimate, using such a sampling process the scaling constant for the decomposition of a random signed permutation into a minimal number of heaps. 

The corresponding problem for the unsigned case has been considered in several papers \cite{istrate2015heapable,istrate2016heapability,basdevant2016hammersley,basdevant2017almost}, and the application of formal power series techniques to the scaling problem is described in \cite{hammersley-rochian}.  
{\small
\begin{figure}[ht]
\begin{center}
	\fbox{
	    \parbox{11cm}{

\textbf{Input:} $k\geq 1,w\in \Gamma_{k}^{*}$ \\
\textbf{Output:} $F_{k}(w)$
\vspace{-2mm}
\begin{itemize} 
\item[] $S:= 0$. $w=w_{1}w_{2}\ldots w_{n}$
\item[] \textbf{if }$w\not \in L(H_k^{sign})$

   \item[]\hspace{5mm}\textbf{return }0 
 
\textbf{if } $w== k^{+}$ or $w==k^{-}$:
   
   \item[]\hspace{5mm} \textbf{return }1 
   
\textbf{for } $i\mbox{ in 1..n}$:
  
       \item[] \hspace{5mm} \textbf{if } $w_{i}==k^{+}\mbox{ and } \exists l: i+1\leq l \leq n: w_{l}\in \{0^{-},\ldots, k^{-}\}$: 
       
               \item[] \hspace{10mm} \textbf{let }$r=min(\{n+1\}\cup \{l\geq i+1: w_{l}\in \{1^{-},\ldots, k^{-}\}) \}$
               \item[] \hspace{10mm}\textbf{ if }$r==n+1$ or $w_r\neq k^{-}$: 
               
               \item[] \hspace{15mm}\textbf{for }$j\mbox{ in 1:r-1}$

               \item[]\hspace{20mm} \textbf{ if }$w_j==0^{-}$:
               \item[] \hspace{25mm} \textbf{       let    }$z=w_{1}\ldots w_{i-1}w_{i+1}\ldots w_{j-1}1^{-}w_{j+1}\ldots       \ldots w_{n}$ 
               \item[] \hspace{25mm} \textbf{       let    }$S := S +  Multiplicity(k,z)$ 
               \item[]\hspace{10mm} \textbf{ if }$r\leq n$:
               \item[] \hspace{15mm} \textbf{       let    }$z=w_{1}\ldots w_{i-1}w_{i+1}\ldots (w_r+1)^{-}w_{r+1}\ldots       \ldots w_{n}$ 
               \item[] \hspace{15mm} \textbf{       let    }$S := S +  Multiplicity(k,z)$
                 \item[]\hspace{5mm} \textbf{if } $w_{i}==k^{-}\mbox{ and } \exists l: i+1\leq l \leq n: w_{l}\in \{0^{+},\ldots, k^{+}\}$:         
       
               \item[]\hspace{10mm} \textbf{let }$r=min(\{n+1\}\cup \{l\geq i+1: w_{l}\in \{1^{+},\ldots, k^{+}\}) \}$
               \item[]\hspace{10mm} \textbf{ if }$r==n+1$ or $w_r\neq k^{+}$: 
               
               \item[]\hspace{15mm} \textbf{for }$j\mbox{ in 1:r-1}$
               \item[]\hspace{20mm} \textbf{ if }$w_j==0^{+}$:
               \item[] \hspace{25mm} \textbf{       let    }$z=w_{1}\ldots w_{i-1}w_{i+1}\ldots w_{j-1}1^{+}w_{j+1}\ldots       \ldots w_{n}$ 
               \item[] \hspace{25mm} \textbf{       let    }$S := S +  Multiplicity(k,z)$ 
               \item[]\hspace{10mm} \textbf{ if }$r\leq n$:
               \item[] \hspace{15mm} \textbf{       let    }$z=w_{1}\ldots w_{i-1}w_{i+1}\ldots (w_r+1)^{+}w_{r+1}\ldots       \ldots w_{n}$ 
               \item[] \hspace{15mm} \textbf{       let    }$S := S +  Multiplicity(k,z)$
  \item[]\hspace{5mm}\textbf{if }$w_{i}== k^{+}$ or $w_i=k^{-}$: 
  
      \item[] \hspace{10mm} \textbf{let }$Z=w_{1}\ldots \ldots w_{i-1}w_{i+1}\ldots w_n$  
      \item[] \hspace{10mm} $S := S +  Multiplicity(k,z)$ 
  
\item[] \textbf{return }S
\end{itemize} 
}
}
\end{center}
\caption{Algorithm Multiplicity(k,w)} 
\label{algo} 
\end{figure}}

\begin{theorem}
Algorithm Multiplicity in Figure~\ref{algo} correctly computes series $F_{k}$. 
\label{foo4} 
\end{theorem}

\begin{proof} The idea of the algorithm is simple, in principle: we simply attempt to "reverse time" and try to identify strings $z$ that can yield the string $w$ in one step of the process. We then add the corresponding multiplicities of all preimages $z$. 

We go from $z$ to $w$ by inserting a $k^{+}$ or a $k^{-}$ and deleting one life from the closest non-zero letter of $z$ that has the correct polarity and is to the right of the newly inserted letter. 

Not all letters of $k^{+},k^{-}$ in $w$ can be candidates for the inserted letter, though: a candidate $k^{+}$, for instance, cannot have a $k^{+}$ as the first letter with positive polarity to its right. That is, if we group the $k^{+}$'s in $w$ in blocks of consecutive occurrences, then the candidate $k^{+}$'s can only be the right endpoints of  a block. 

As for the letter $l$ the new $k^{+}$ acted upon in $z$, in $w$ it cannot be a $k$; also, there cannot be any letters with positive polarity, other than zero, between the new $k^{+}$ and $l$. So $l$ is either one of the $0^{+}$'s at the right of the new $k^{+}$ or the first nonzero positive letter (if not $k^{+}$). It is also possible that $l$ does not exist.  

Analogous considerations apply if the newly inserted letter is a $k^{-}$. 
\end{proof}\qed

\section{Motivation: The Ulam-Hammersley problem for the heap decomposition of signed permutations} 

The \textit{Ulam-Hammesley problem} \cite{ulam1961monte,hammersley1972few} can be described as follows: estimate the asymptotic behavior of the expected length of the longest increasing subsequence of a random permutation $\sigma \in S_n$. The correct scaling is $E_{\sigma\in S_{n}}[LIS[\sigma]]=2\sqrt{n}(1+o(1))$, however substantially more is known, and the problem has deep connection with several mathematical concepts and areas, including Young tab-\\ leaux (see e.g.  \cite{romik2015surprising}), random matrix theory \cite{baik2016combinatorics} and interacting particle systems \cite{aldous1995hammersley}. 

The following concept has been defined (for $k=2$, and can be easily generalized as presented) in \cite{byers2011heapable}: a sequence of integers is \textit{($k$)-heapable} if the elements of a sequences can be inserted successively in a $k$-ary tree, min-heap ordered (i.e. the label of the child is at least as large as that of the parent), not-necessarily complete, so that insertions are always made as a leaf.  

Heapable sequences can be viewed as a (loose) generalization of increasing sequences: instead of inserting sequences into an increasing array, so that each node (except the last one) has exactly one successor, we allow a "more relaxed version" of this data structure, in the form of a $k$-ary tree with a min-heap ordering on the nodes. In other words a node has $k$-choices for the insertion, and using all these leaves can accommodate some limited form of disorder of consecutive elements, although values in such a sequence "tends to increase", simply because the available positions may tend to be lower and lower in the tree. 

 It is natural, therefore, to attempt to generalize the Ulam-Hammersley problem to heapable sequences. The direct extension is problematic, though: as discussed in \cite{byers2011heapable}, the problem of computing the longest heapable subsequence has unknown complexity (see \cite{chandrasekaran2020fixed,chandrasekaran2021maximum} for some related results).  A more promissing alternative is the following: by Dilworth's theorem the longest increasing subsequence is equal to the minimal number of \textit{decreasing sequences} in which we can partition the sequence. So the "correct" extension of the longest increasing subsequence is (see \cite{istrate2015heapable,balogh2017computing}): \textbf{decompose a random permutation $\sigma\in S_n$ into the minimal number of heap-ordered $k$-ary trees}, and study the scaling of the expected number of trees in such an optimal  decomposition. 

The longest increasing subsequence problem has also been studied for \textit{signed} or even \textit{colored} permutations \cite{borodin1999longest}. It is reasonable to consider a similar problem for heapability. This is what we do in this paper. Note, though, that the variant we consider here is not the same as the one in \cite{borodin1999longest}. Specifically, in \cite{borodin1999longest} the author requires that the "colors" (i.e. signs) of two adjacent nodes that are in a parent-child relationship are the same. By contrast we require that \textit{the colors of any two adjacent nodes are different.}

\begin{definition} 
A \textit{signed permutation of order $n$} is a pair $(\sigma,\tau)$, with $\sigma$ being a permutation of size $n$ and $\tau: [n]\rightarrow \{\pm 1\}$ being a \textit{sign function}. 
\end{definition}

\begin{definition}  Given integer $k\geq 1$, signed permutation $(\sigma,\tau)$ is called $\leq k$-\textit{heapable} if one can successively construct  $k$ (min) heap-ordered binary trees (not necessarily complete) $H_{0},H_{1},\ldots, H_{k-1}$ such that 
\begin{itemize} 
\item at time $i=0$ $H_{0},H_{1},\ldots H_{k-1}$ are all empty. 
\item for every $1\leq i\leq n$ element $\sigma(i)$ can be inserted as a new leaf in one of $H_{0},H_{1},\ldots,$ $ H_{k-1}$, such that if $\sigma(j)$, the parent of $\sigma(i)$ exists, $j<i$, then ($\sigma(j)<\sigma(i)$ and) $\tau[i]=-\tau[j]$. 
\end{itemize} 

$(\sigma,\tau)$ will be called \textit{$k$-heapable} iff $k$ is the smallest parameter such that $(\sigma,\tau)$ is $\leq k$-heapable. We will write \textit{heapable} instead of $1$-heapable. 
\end{definition} 

First we note that heapability of signed permutations is not a simple extension of heapability of ordinary permutations: 

\begin{observation} 
Heapability of permutation $\sigma$ is \textit{not} equivalent to heapability of any fixed signed-version of $\sigma$. In particular if $\tau_{0}[i]=1$ for all $i=1, \ldots, n$ then every signed permutation $(\sigma,\tau_{0})$ is $n$-heapable (even though $\sigma$ might be heapable as a permutation). This is because the sign condition forces all elements of $\sigma$ to start new heaps, as no two values have opposite signs. 
\end{observation} 

On the other hand a connection between ordinary heapability and that of signed permutations does exist after all: 

\begin{theorem} The following are true: 
\begin{itemize} 
\item If $(\sigma,\tau)$ is heapable (as a signed permutation) then $\sigma$ is heapable (as an ordinary permutation). 
\item There is a polynomial time algorithm that takes as input a permutation $\sigma\in S_n$ and produces a sign $\tau:[n]
\rightarrow \{\pm 1\}$ such that if $\sigma$ is heapable (as an ordinary permutation) then $(\sigma,\tau)$ is heapable (as a signed permutation). 
\end{itemize} 
\end{theorem} 

\begin{proof} 

\begin{itemize} 
\item Trivial, since ordinary heapability does not care about signing restrictions. 
\item A simple consequence of the greedy algorithm for (ordinary) heapability \cite{byers2011heapable}: roughly speaking, given a permutation $\sigma$ there is a canonical way of attempting to construct a min-heap for $\sigma$: each element $\sigma(i)$ is added as a child of  the largest element $\sigma(j)<\sigma(i)$, $j<i$ that still has a free slot. If no such $\sigma(j)$ is available then $\sigma$ is not heapable. Further, denote by $p[i]$ the index of $\sigma(j)$ in the permutation $\sigma$ (i.e. $p[i]=j$). 

We use this algorithm to construct signing $\tau$ inductively. Specifically, set $\tau(1)=1$. Given that we have constructed $\tau(j)$ for all $j<i$, define $\tau(i)=-\tau(p[i])$. Then $(\sigma,\tau)$ is heapable as a signed permutation, since the greedy solution for $\sigma$ is legal for $(\sigma,\tau)$. 
\end{itemize} 
\end{proof}\qed

\subsection{A greedy algorithm for the optimal heap decomposition of signed permutations} 

We now give a greedy algorithm for decomposing signed permutations into a minimal number of heap-ordered $k$-ary binary trees. The algorithm is presented in Figure~\ref{fig-alg-greedy} and is based on the concept of \textit{slots}. The algorithm ends by creating an forest of heaps based on the initial signed integer sequence. For each number $x_i$ the algorithm at step $i$ even will add it to an existing heap or will start a new heap with it. When it can be added on an existing heap the sign of it is very important. We search for the highest value less than $x_i$ \textrm{having opposite sign to $\sigma(i)$}, and will be created $k$ \textit{slots of value $\sigma(i)$}, allowing the subsequent insertion of values in the interval $[\sigma(i), \infty)$. 

\begin{example}
To better explain this concept, consider the sequence of insertions in a heap-ordered tree for a permutation $(\sigma,\tau)$ with initial prefix $\sigma(1)=1,\tau(1)=-1$, $\sigma(2)=8, \tau(2)=+1$, $\sigma(3)=5, \tau(3)=-1$, or shortly $\sigma=\{1, 8, 5\}$ and $\tau=\{-, +, -\}$ and $k=2$.  There is essentially an unique way (displayed in Figure~\ref{ex:NodesAndSlots}) to insert these elements successively into a single  heap: $\sigma(1)$ at the root, $\sigma(2)$ as a child of $\sigma(1)$, $\sigma(3)$ as a child of $\sigma(2)$. Note that $\sigma(3)$ needs to be a child of $\sigma(2)$ since the condition $\tau(3)=-1$ forbids, e.g. placing $\sigma(3)$ as a child of $\sigma(1)$. 
\label{expl-trees} 
 \end{example} 
 
 \begin{figure}[ht]
 \begin{center} 
    \tikzset{
      basic/.style  = {draw, text width=1em, drop shadow,  circle},
      nod/.style   = {basic, thin, align=center, fill=gray!45},
      slot/.style = {rectangle,thin,align=center, fill=gray!60}
    }
    \begin{minipage}{.3\textwidth}
    \begin{tikzpicture}[
     level 1/.style={sibling distance=20mm},
      edge from parent/.style={->,draw},
      level 2/.style={sibling distance=20mm},
      >=latex]
     root of the the initial tree, level 1
    \node[nod] {$1$}
      child {node[slot] (c4) {$[1,\infty)^+$}}
      child {node[slot] (c2) {$[1,\infty)^+$}}
   ; 
    \end{tikzpicture}
    \end{minipage}  
   \begin{minipage}{.3\textwidth}
   \begin{tikzpicture}[,
      level 1/.style={sibling distance=20mm},
      edge from parent/.style={->,draw},
      level 2/.style={sibling distance=20mm},
      >=latex]
    \node[nod] {$1$}
      child {node[nod] (c6) {$8$}
      			child {node[slot] (c81) {$[8,\infty]^-$}}
				child {node[slot] (c82) {$[8,\infty]^-$}}
			}
	child {node[slot] (c10) {$[1,\infty)^+$}}	
    ;
    \end{tikzpicture}
    \end{minipage} 
   \begin{minipage}{.3\textwidth}
   \begin{tikzpicture}[,
      level 1/.style={sibling distance=20mm},
      edge from parent/.style={->,draw},
      level 2/.style={sibling distance=20mm},
      >=latex]
    \node[nod] {$1$}
      child {node[nod] (c6) {$8$}
      			child {
      				node[nod] {$15$}
      				child {node[slot] (c4) {$[15,\infty]^+$}}
      				child {node[slot] (c2) {$[15,\infty]^+$}}
      			}
				child {node[slot] (c82) {$[8,\infty]^-$}}
			}
	child {node[slot] (c10) {$[1,\infty)^+$}}	
    ;
    \end{tikzpicture}
    \end{minipage} 
    \end{center} 
    \caption{Nodes and slots.}
    \label{ex:NodesAndSlots}
  \end{figure}
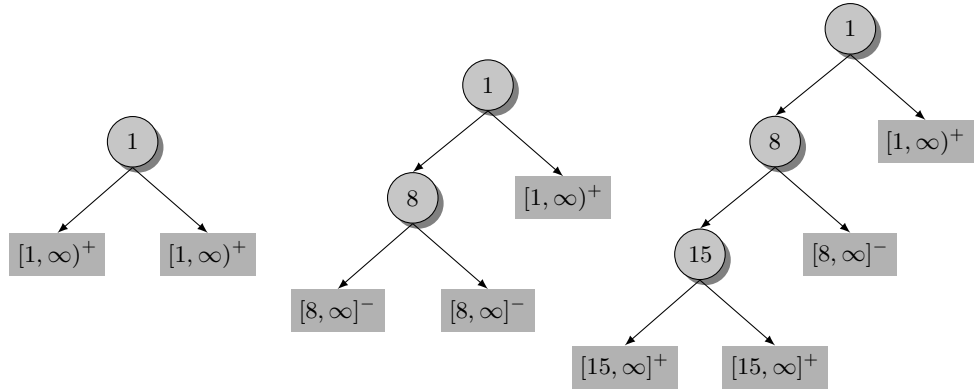

Each insertion creates two \textit{slots}, having an interval as values. Intuitively the interval of a slot describes the constraints imposed on an integer to be inserted in that position in order to satisfy the heap constraint. For example, after inserting $\sigma(2)$ the unique free slot of the root has value $1$, since  all children of the root must have values bigger than $\sigma(1)=1$. The unique free slot of node hosting $\sigma(2)$ has value $8$, since any descendant of this node must have minimum value $8$. 

Our main result of this section is: 

\begin{theorem} 
The algorithm GREEDY, presented in Fig.~\ref{fig-alg-greedy}, decides the heapability of an arbitrary signed permutation and, more generally, constructs an optimal heap decomposition of $(\sigma,\tau)$. 
\label{sign-decomposition} 
\end{theorem} 

\begin{figure} 
\begin{center}
\begin{pseudocode}[doublebox]{GREEDY}{\sigma,\tau}
\mbox{INPUT } \sigma=(\sigma(1),\sigma(2),\ldots, \sigma(n))\mbox{ a permutation in }S_n\mbox{ and} \\
\hspace{11mm}\tau=(\tau(1),\tau(2),\ldots, \tau(n))\mbox{ a list of \{+,-\} signs} \\
\\
\mbox{start with empty heap forest }T=\emptyset. \\
\\
\mbox{for }i\mbox{ in range(n):} \\
	\hspace{5mm}\IF \mbox{(there exists a slot where }\sigma(i)\mbox{ can be inserted, according to } \tau(i) \mbox{):}\\
	\hspace{20mm}\mbox{ insert }\sigma(i)\mbox{ in the slot with the largest compatible value }
	\hspace{20mm} \ELSE :\\
	\hspace{20mm}\mbox{ start a new heap consisting of }\sigma(i)\mbox{ only.} 
\end{pseudocode}
\end{center}
\caption{The greedy algorithm for decomposing a signed permutation into a forest of heaps.}
\label{fig-alg-greedy}
\end{figure} 
\begin{proof} 

Proving  correctness of the algorithm employs the following 

\begin{definition} 
Given a binary heap forest $T$, define the \textit{positive signature of $T$} denoted $sig^{+}(T)$, to be the vector containing the (values of) free slots with positive polarity in $T$, sorted in non-decreasing order. The negative signature of $T$, denoted by $sig^{-}(T)$, is defined analogously. 

Given two binary heap forests $T_{1},T_{2}$, {\it $T_{1}$ dominates $T_{2}$} if 
\begin{itemize} 
\item $|sig^{+}_{T_{1}}|\leq |sig^{+}_{T_{2}}|$ and inequality $sig^{+}_{T_{1}}[i]\leq sig^{+}_{T_{2}}[i]$ holds for all $1\leq i \leq |sig^{+}_{T_{1}}|$.
\item Similarly for negative signatures. 
\end{itemize} 
\end{definition}

\begin{lemma} 
Let $T_{1},T_{2}$ be two heap-order forests  such that $T_{1}$ dominates $T_{2}$. Insert a new element $x$ greedily in  $T_{1}$ (i.e. at the largest slot with value less or equal to $x$, or as the root of a new tree, if no such slot exists). Also insert $x$ into an arbitrary compatible slot in $T_{2}$. These two insertions yield heap-ordered forests $T_{1}^{\prime},T_{2}^{\prime}$, respectively. 
Then $T_{1}^{\prime}$ dominates $T_{2}^{\prime}$.
\label{dom}
\end{lemma} 
\begin{proof} 

We need the following definition 

\begin{definition}
Given integer $k\geq 1$, a \textit{$k$-multiset} is a multiset $A$ such that each element of $A$ has multiplicitly at most $k$, that is a function $f:A\rightarrow \{0,1,\ldots, k\}$. 
\end{definition} 

\begin{definition} Given multiset $A$, a \textit{Hammersley insertion of an element $x$ into $A$} is the following process: 
\begin{itemize} 
\item $x$ is given multiplicity $k$ in $A$. 
\item Some element $y\in A$, $y>A$ (if any exists), has its multiplicity reduced by 1. 
\end{itemize} 
It is a \textit{greedy Hammersley insertion} if (when it exists) $y$ is the smallest element of $A$ larger than $x$ having positive multiplicity. 
\end{definition} 

\begin{definition} Given two $k$-multisets $A_1$ and $A_2$ of elements of $\Gamma_{k}$, we define $A_{i}^{+}, A_{i}^{-}$ the submultisets of $A_i$, $i=1,2$ consisting of elements with positive (negative) polarity only. 

We say that $A_1$ dominates $A_2$ iff: 
\begin{itemize} 
\item $|A^{+}_{1}|\leq |A^{+}_{2}|$ and inequality $A^{+}_{1}[i]\leq A^{+}_{2}[i]$ holds for all $1\leq i \leq |A^{+}_{1}|$.
\item Similarly for negative polarities.  
\end{itemize} 

\end{definition} 

The following result was not explicitly stated in \cite{istrate2015heapable} but was implicitly proved, as the basis of the proof of Lemma 1 from \cite{istrate2015heapable} (the analog of Lemma~\ref{dom} for the unsigned case): 

\begin{lemma} If $A_1$ and $A_2$ are multisets of integers such that $A_1$ dominates $A_2$. Consider $x$ an element not present in either set and let $A_{1}^{\prime}$, 
$A_{2}^{\prime}$ be the result of greedy Hammersley insertion into $A_1$ and an arbitrary Hammersley insertion into $A_2$. 

Then $A_{1}^{\prime}$ dominates $A_{2}^{\prime}$. 
\label{dom-1}
\end{lemma} 

Rather than repeating the proof, we refer the reader to \cite{istrate2015heapable}. 
To be able to prove Lemma~\ref{dom} we also need the following: 

\begin{lemma} The following are true: 
\begin{itemize} 
\item[-] Let $A_1$ and $A_2$ be $k$-multisets of integers and let $x$ be an integer that does not appear in $A_1,A_2$. Then, if $A_{1}^{\prime},A_{2}^{\prime}$ represent the result of inserting $x$ with multiplicity $k$ in $A_1,A_{2}$, respectively, then $A_{1}^{\prime}$ dominates $A_{2}^{\prime}$. 
\item Let $A_1$ and $A_2$ be $k$-multisets of integers and let $x$ be an integer that appears in both $A_1,A_2$ with the same multiplicity. Then, if $A_{1}^{\prime},A_{2}^{\prime}$ represent the result of deleting $x$ from $A_1,A_{2}$, respectively, then $A_{1}^{\prime}$ dominates $A_{2}^{\prime}$.
\end{itemize}
\label{dom-2} 
\end{lemma} 
\begin{proof} 

\begin{itemize} 
\item[-] It is clear that $x$ adds $k$ to the cardinality of both $A_1,A_2$. To prove domination we have to, therefore, prove that $A_{1}^{\prime}[i]\leq A_{2}^{\prime}[i]$ for all $1\leq i\leq |A_{1}^{\prime}|$. Let $i_1$ be the smallest index such that $A_{1}(i_1)> x$ and $i_2$ be the smallest index such that $A_{2}(i_2)> x$. Because $A_1$ dominates $A_{2}$, $i_1\geq i_2$. 
\begin{itemize} 
\item{Case 1: $i< i_2$.} Then $A_{1}^{\prime}[i]=A_{1}[i]$ and $A_{2}^{\prime}[i]=A_{2}[i]$. Inequality follows from the fact that $A_1$ dominates $A_2$. 
\item{Case 2: $i\geq i_1+k$.} Then $A_{1}^{\prime}[i]=A_{1}[i-k]$ and $A_{2}^{\prime}[i]=A_{2}[i-k]$. Inequality follows from the fact that $A_1$ dominates $A_2$.
\item{Case 3: $i_2\leq i<i_1$:} Then $A_{1}^{\prime}[i]=A_{1}[i]<x$ and $A_{2}^{\prime}[i]=x$. So $A_{1}^{\prime}[i]\leq A_{2}^{\prime}[i]$. 
\item{Case 4: $i_1\leq i<i_1+k$:} Then $A_{1}^{\prime}[i]=x$ and $A_{2}^{\prime}[i]\geq A_{2}^{\prime}[i_1]\geq A_{2}^{\prime}[i_2]=x$. So $A_{1}^{\prime}[i]\leq A_{2}^{\prime}[i]$. 
\end{itemize} 
\item[-] It is clear that $x$ subtracts $k$ to the cardinality of both $A_1,A_2$. To prove domination we have to, therefore, prove that $A_{1}^{\prime}[i]\leq A_{2}^{\prime}[i]$ for all $1\leq i\leq |A_{1}^{\prime}|$. Let $i_1$ be the smallest index such that $A_{1}(i_1+k)> x$ and $i_2$ be the smallest index such that $A_{2}(i_2+k)> x$. $i_1,i_2$ are well-defined since $A_1,A_2$ contain $k$ copies of $x$. Because $A_1$ dominates $A_{2}$, $i_1\geq i_2$. The first position of $x$ in $A_1$ is $i_1$, the last is $i_1+k-1$; the first position of $x$ in $A_2$ is $i_2$, the last is $i_2+k-1$. 
\begin{itemize} 
\item{Case 1: $i< i_2$.} Then $A_{1}^{\prime}[i]=A_{1}[i]$ and $A_{2}^{\prime}[i]=A_{2}[i]$. Inequality follows from the fact that $A_1$ dominates $A_2$. 
\item{Case 2: $i\geq i_1$.} Then $A_{1}^{\prime}[i]=A_{1}[i-k]$ and $A_{2}^{\prime}[i]=A_{2}[i-k]$. Inequality follows from the fact that $A_1$ dominates $A_2$.
\item{Case 3: $i_2\leq i<i_1$:} Then $A_{1}^{\prime}[i]=A_{1}[i]<x$ and $A_{2}^{\prime}[i]=A_{2}[i+k]>x$. So $A_{1}^{\prime}[i]\leq A_{2}^{\prime}[i]$. 
\item{Case 4: $i_1\leq i<i_1+k$:} Then $A_{1}^{\prime}[i]<x$ and $A_{2}^{\prime}[i]\geq A_{2}^{\prime}[i_1]\geq A_{2}^{\prime}[i_2]=A_{2}[i_2+k]>x$. So $A_{1}^{\prime}[i]\leq A_{2}^{\prime}[i]$. 
\end{itemize} 
\end{itemize} 
\end{proof}\qed

Now the proof of Lemma~\ref{dom} follows: the effect of inserting an element with positive polarity $x^{+}$ greedily into $T_1$ can be simulated as follows: 
\begin{enumerate} 
\item[-] perform a greedy Hammersley insertion of $x^-$ into $sig^{-}(T_1)$. 
\item[-] remove $x^-$ from $sig^{-}(T_1)$, and insert $x^+$ into $sig^{+}(T_1)$. 
\end{enumerate} 
On the other hand we can insert $x^{+}$ into $T_2$ as follows: 
\begin{enumerate} 
\item[-] perform a Hammersley insertion of $x^-$ into $sig^{-}(T_2)$. 
\item[-] remove $x^-$ from $sig^{-}(T_2)$, and insert $x^+$ into $sig^{+}(T_2)$. 
\end{enumerate} 
By applying Lemma~\ref{dom-1} and~\ref{dom-2} we infer that after the insertion of $x^{+}$ $T_{1}^{\prime}$ dominates $T^{\prime}_{2}$. The insertion of an element with negative polarity is similar. 
\end{proof} 
\qed 

Using Lemma~\ref{dom} we can complete the proof of Theorem~\ref{sign-decomposition} as follows: 
by domination, whenever no slot of $T_{1}$ can accommodate $x$ (which, thus, starts a new tree) then the same thing happens in $T_{2}$ (and thus $x$ starts a new tree in $T_{2}$ as well). So the greedy algorithm is at least as good as any sequence of insertions, which means it is optimal. 
\qed

\end{proof}\qed 

\subsection{Connection with the Signed Hammersley process}

We are now in a position to explain what role does the signed Hammersley process play in the Ulam-Hammersley problem for the heap decomposition of signed permutations: to each heap-ordered forest $F$ associate a word $w_{F}$ over $\Gamma_{k}^{*}$ which describes the relative positions of slots in $F$. Specifically, sort the leaves of $F$ according to  their value. If leaf $f$ has at a certain moment $p\leq k$ slots of, say, negative polarity, then encode this into $w_{F}$ by letter $p^{-}$.

\begin{example} Let us consider the trees of Example~\ref{expl-trees}. For the first tree the associated word is $2^{+}$. For the second tree the word is $1^{+}2^{-}$. For the third tree the word is $1^{+}1^{-}2^{+}$. 
\end{example} 

\begin{observation} 
Note that \textrm{when inserting a new element into the heap forest using the greedy algorithm the associated word changes exactly according to the signed Hammersley process\footnote{or, rather, the signed Hammersley process that removes a one to the left, rather than to the right. We would have to use min-heaps to obtain the signed Hammersley process. But this change is inconsequential.}}. 
In fact we can say more: starting a new heap-ordered tree corresponds to moments when we insert a new $k^{+}$ or $k^{-}$ (whichever is appropriate at the given moment) without subtracting any 1 from the current word.  
So, if $z$ is the outcome of a sequence of greedy insertions then $trees(z)$, the number of heap-ordered $k$-ary trees created in the process  is equal to $\lambda^{+}+\lambda^{-}$ (in the notation of equations~\ref{ineq-3} and~\ref{ineq-4}), and is equal to 
\[
trees_{k}(z)=|z|_{k}-\sum\limits_{i=1}^{k} i\cdot |z|_{k-i}+\sum\limits_{i=1}^{k}|z|_{k-i}=|z|_{k}-\sum\limits_{i=1}^{k}(i-1)|z|_{k-i}
\]
Of course, each word $z$ may arise with a different multiplicity in the signed Hammersley process. So to compute the expected number of heap-ordered $k$-ary trees generated by using the greedy algorithm with a random signed permutation of length $n$ as input we have to compute amount
\begin{equation} 
Z_{n}^{k}:=\frac{1}{2^{n}\cdot n!}\sum_{z\in (\Gamma^{k})^{n}} F_{k}(z)\cdot trees_{k}(z).
\label{scaling-trees}  
\end{equation} 
\label{obs-final} 
\end{observation}  

The last equation in Observation~\ref{obs-final} is what motivated us to give Algorithm~\ref{algo} for computing the formal power series $F_{k}$. We defer, however, the problem of experimentally investigating the scaling behavior of $Z_{n}^{k}$ as $n\rightarrow \infty$ using Algorithm~\ref{algo} to subsequent work. 

\section{Conclusions} 

The main contribution of this paper is to show that a significant number of  analytical tools developed for the analysis of the Ulam-Hammersley problem for heapable sequences \cite{istrate2015heapable,hammersley-rochian} extend to the case of signed permutations. 
While going along natural lines, the extension has some moderately interesting features: for instance the languages in the sign case appear to have slightly higher grammatical complexity than those in \cite{hammersley-rochian} for the ordinary process.  

The obvious continuation of our work is the investigation of the Ulam-Hammersley problem for signed (and, more generally, colored) permutations. We do not mean only the kind of experiments alluded to in Observation~\ref{obs-final}. For instance it is known that the scaling in the case of ordinary permutations is logarithmic \cite{basdevant2016hammersley,basdevant2017almost}, even though the proportionality constant is not rigorously known. Obtaining similar results for the signed Hammersley (tree) process would be, we believe, interesting. 

On the other hand, as we noted, our extension to signed permutations is not a direct version of the one in \cite{borodin1999longest}. Studying a variant of our problem consistent with the model in \cite{borodin1999longest} (or studying the longest increasing subsequence in the model we consider) is equally interesting.

\bibliographystyle{psc}
\bibliography{/Users/gistrate/Dropbox/texmf/bibtex/bib/bibtheory}

\begin{thebibliography}{10}

\bibitem{aldous1995hammersley}
{\sc D.~Aldous and P.~Diaconis}:
\newblock \textit{ Hammersley's interacting particle process and longest increasing
  subsequences}.
\newblock Probability theory and related fields, 103(2) 1995, pp.~199--213.

\bibitem{baik2016combinatorics}
{\sc J.~Baik, P.~Deift, and T.~Suidan}:
\newblock \textit{ Combinatorics and random matrix theory}, vol.~172, American
  Mathematical Soc., 2016.

\bibitem{balogh2017computing}
{\sc J.~Balogh, C.~Bonchi{\c{s}}, D.~Dini{\c{s}}, G.~Istrate, and I.~Todinca}:
\newblock \textit{ On the heapability of finite partial orders}.
\newblock Discrete Mathematics and Theoretical Computer Science, 22(1) 2020, paper \# 17.

\bibitem{basdevant2016hammersley}
{\sc A.-L. Basdevant, L.~Gerin, J.-B. Gou{\'e}r{\'e}, and A.~Singh}:
\newblock \textit{ From {H}ammersley's lines to {H}ammersley's trees}.
\newblock Probability Theory and Related Fields,  2016, pp.~1--51.

\bibitem{basdevant2017almost}
{\sc A.-L. Basdevant and A.~Singh}:
\newblock \textit{ Almost-sure asymptotic for the number of heaps inside a random
  sequence}.
\newblock Electronic Communications in Probability, 23(17) 2018.

\bibitem{berstel2011noncommutative}
{\sc J.~Berstel and C.~Reutenauer}:
\newblock \textit{ Noncommutative rational series with applications}, vol.~137,
  Cambridge University Press, 2011.

\bibitem{hammersley-rochian}
{\sc C.~Bonchi\c{s}, G.~Istrate, and V.~Rochian}:
\newblock \textit{ The language (and series) of {H}ammersley-type processes}, in
  Proceedings of the 8th Conference on Machines, Computation and Universality
  (MCU'18), vol.~10881 of Lecture Notes in Computer Science, 2018.

\bibitem{borodin1999longest}
{\sc A.~Borodin}:
\newblock \textit{ Longest increasing subsequences of random colored permutations}.
\newblock The Electronic Journal of Combinatorics, 6(1) 1999, p.~R13.

\bibitem{byers2011heapable}
{\sc J.~Byers, B.~Heeringa, M.~Mitzenmacher, and G.~Zervas}:
\newblock \textit{ Heapable sequences and subseqeuences}, in Proceedings of the
  Eighth Workshop on Analytic Algorithmics and Combinatorics (ANALCO'2011),
  SIAM Press, 2011, pp.~33--44.

\bibitem{chandrasekaran2020fixed}
{\sc K.~Chandrasekaran, E.~Grigorescu, G.~Istrate, S.~Kulkarni, Y.-S. Lin, and
  M.~Zhu}:
\newblock \textit{ Fixed-parameter algorithms for longest heapable subsequence and
  maximum binary tree}, in 15th International Symposium on Parameterized and
  Exact Computation, 2020.

\bibitem{chandrasekaran2021maximum}
{\sc K.~Chandrasekaran, E.~Grigorescu, G.~Istrate, S.~Kulkarni, Y.-S. Lin, and
  M.~Zhu}:
\newblock \textit{ The maximum binary tree problem}.
\newblock Algorithmica, 83(8) 2021, pp.~2427--2468.

\bibitem{hammersley1972few}
{\sc J.~M. Hammersley et~al.}:
\newblock \textit{ A few seedlings of research}, in Proceedings of the Sixth
  Berkeley Symposium on Mathematical Statistics and Probability, Volume 1:
  Theory of Statistics, 1972.

\bibitem{harrison1978introduction}
{\sc M.~A. Harrison}:
\newblock \textit{ Introduction to formal language theory}, Addison-Wesley Longman
  Publishing Co., Inc., 1978.

\bibitem{istrate2015heapable}
{\sc G.~Istrate and C.~Bonchi\c{s}}:
\newblock \textit{ Partition into heapable sequences, heap tableaux and a multiset
  extension of {H}ammersley's process}, in Proceedings of the 26th Annual
  Symposium on Combinatorial Pattern Matching (CPM'15), Ischia, Italy,
  vol.~9133 of Lecture Notes in Computer Science, Springer, 2015, pp.~261--271.

\bibitem{istrate2016heapability}
{\sc G.~Istrate and C.~Bonchi{\c{s}}}:
\newblock \textit{ Heapability, interactive particle systems, partial orders:
  Results and open problems}, in Proceedings of the 18th International
  Conference on Descriptional Complexity of Formal Systems (DCFS'2016),
  Bucharest, Romania, vol.~9777 of Lecture Notes in Computer Science, Springer,
  2016, pp.~18--28.

\bibitem{heapability-thesis}
{\sc J.~Porfilio}:
\newblock \textit{ A combinatorial characterization of heapability}, Master's
  thesis, Williams College, May~2015, available from
  \url{https://unbound.williams.edu/theses/islandora/object/studenttheses\%3A907}.
  Accessed: May 2024.

\bibitem{romik2015surprising}
{\sc D.~Romik}:
\newblock \textit{ The surprising mathematics of longest increasing subsequences},
  Cambridge University Press, 2015.

\bibitem{salomaa2012automata}
{\sc A.~Salomaa and M.~Soittola}:
\newblock \textit{ Automata-theoretic aspects of formal power series}, Springer
  Science \& Business Media, 2012.

\bibitem{ulam1961monte}
{\sc S.~M. Ulam}:
\newblock \textit{ Monte carlo calculations in problems of mathematical physics}.
\newblock Modern Mathematics for the Engineers, 261 1961, p.~281.

\bibitem{valiant1975deterministic}
{\sc L.~G. Valiant and M.~S. Paterson}:
\newblock \textit{ Deterministic one-counter automata}.
\newblock Journal of Computer and System Sciences, 10(3) 1975, pp.~340--350.

\end{thebibliography}

\end{document}